\documentclass[12pt]{amsart}
\usepackage[utf8]{inputenc}
\usepackage{mathtools,amssymb}
\usepackage{xcolor}
\usepackage{bm}
\usepackage[margin=0.75in,top=0.8in]{geometry}
\everymath{\displaystyle}

\newtheorem{theorem}{Theorem}[section]

\newtheorem{corollary}{Corollary}[section]
\newtheorem{definition}{Definition}[section]

\newtheorem{lemma}{Lemma}[section]

\newtheorem{remark}{Remark}[section]

\usepackage{hyperref}

\title{ Weak Observability Characterization for Abstract  Wave Equations } 

\author{Nisrine Charaf\dag}
\address{ \dag Laboratoire Jean Kuntzmann,  UMR CNRS 5224, 
Universit\'e  Grenoble-Alpes, 
Bat. IMAG, 150 Pl. du Torrent 38400, St Martin d’H\`eres, France.}
\email{Nisrine.Charaf@univ-grenoble-alpes.fr}

\author{Faouzi Triki\ddag}
\address{\ddag Laboratoire Jean Kuntzmann,  UMR CNRS 5224, 
Universit\'e  Grenoble-Alpes, 
Bat. IMAG, 150 Pl. du Torrent 38400, St Martin d'H\`eres, France.}
\email{faouzi.triki@univ-grenoble-alpes.fr}

\begin{document}

\begin{abstract}
In this paper, we investigate the weak observability of second-order infinite-dimensional evolution systems generated by skew-adjoint operators of the form $iA_0$, where $A_0$ is a self-adjoint elliptic operator. We first establish a spectral characterization of weak observability by introducing the notion of spectral coercivity for the observation operator and proving its equivalence to a suitable resolvent estimate. Our main result reveals a direct link between resolvent estimates for the elliptic operator $A_0$ and the weak observability of the associated evolution generator $A$. More precisely, we prove that a resolvent inequality for $A_0$ implies a Hautus-type spectral observability estimate for $A$, which guarantees the weak observability of the system. This provides a unified spectral framework for weak observability based on the coercivity properties of the observation operator. As an application, we establish explicit weak observability estimates for the wave equation on a rectangular domain under several geometric configurations of the observation region. The analysis combines frequency-domain methods, resolvent estimates, and Fourier analysis, yielding new insights into the interplay between resolvent inequalities, spectral coercivity, and weak observability in infinite-dimensional systems.

\end{abstract}

\maketitle

\section{Introduction}\label{observability: section1}
\noindent

Let $H$ be a complex Hilbert space endowed with the norm and inner product respectively ${\lVert . \rVert_H}$ and $\langle \cdot, \cdot \rangle_H$. Let $A_{0}:\, H \rightarrow H$ be a linear unbounded self-adjoint positive operator with a compact resolvent and with domain $D(A_{0})$.  Denote by $D(A_0^{\frac{1}{2}})$ the domain of $A_0^{\frac{1}{2}}$, and define for $\beta$ $\in$ $\mathbb{R}$ the scale of Hilbert spaces $H_\beta$ by  $H_{\beta} = D(A_0^{\frac{\beta}{2}})$ with the norm $ \lVert u \rVert_{H_{\beta}} = \lVert A_0^{\frac{\beta}{2}}u \rVert_{H}$ for $\beta \geq 0$, and by $H_\beta= H^*_{-\beta}$ using the duality with respect to the pivot space $H$  for $\beta<0$. \\
The operator $A_0$ can be extended 
or restricted to $H_\beta$ to becomes a bounded linear operator $A_0 : \, H_\beta \rightarrow H_{\beta-2}$ for all
$\beta \in \mathbb R$.\\
\noindent  Let $Y$ be a complex Hilbert space equipped with the norm and scalar product respectively ${\lVert \cdot \rVert_{Y}}$ and $\langle \cdot, \cdot \rangle_Y$.
Denote $  \mathcal L(D(A_{0}), Y)$  the space of linear bounded operators from $D(A_{0})$ into $Y$, and let
$C_{0} \in  \mathcal L(D(A_{0}), Y)$.  \\

\noindent  This paper is concerned with the following abstract infinite-dimensional dual observation system described by the equations

\begin{equation}\label{waveequationn}
\left\{ \begin{array}{lllcc}
   \ddot u(t)+A_0 u(t) = 0, \quad  t>0, \\
    w(t) = C_0u(t), \quad t>0, \\
    u(0) = u_{0},\; \dot u(0) = u_{1}.
\end{array}\right.
\end{equation}
\noindent Here, the dot notation represents differentiation with respect to the time $t$. The element $z_0 =(u_0, u_1)^t $ is referred to as the $initial \  state$, while $u(t)$ denotes the state at time $t$ and $w$ is the output function. Such systems are commonly used to model vibrating phenomena  involving  the scalar or vector  wave equation that arise  in fields like acoustics, electromagnetism, fluid dynamics and mechanics.\\

\noindent We further consider the following operator-valued matrices 

\begin{eqnarray} A = -i \begin{pmatrix} 0 & I \\ - A_0 & 0 \end{pmatrix}, \quad  C = \begin{pmatrix} C_{0} & 0 \end{pmatrix}. \label{operatorsAC}
\end{eqnarray}

Let $X = H_1 \times H$ be the product of the complex Hilbert spaces $H_1$ and $H$. We equip $X$ with the canonical product norm $\|\cdot\|_X$ and the corresponding inner product $\langle \cdot,\cdot\rangle_X$.
Therefore, $A :  X \rightarrow X$ is a linear unbounded self-adjoint  operator with a compact resolvent, with domain $D(A)= D(A_0)\times H_1$, and $C \in  \mathcal L(D(A), Y)$. According to Stone's theorem, $iA$ (resp. $iA_0$) generates a strongly continuous group of isometries in $X$ (resp. $H$) denoted $(e^{itA})_{t {\in \mathbb{R}} }$  (resp. 
$(e^{itA_0})_{t {\in \mathbb{R}} }$)\cite{TW}. \\

\noindent The second order evolution system \eqref{waveequationn} can be expressed in the following first order evolution system  

\begin{equation}\label{observability1}
\left\{ \begin{array}{lllcc}
   \dot {z} (t)  -i A z(t) = 0 , \quad  t>0, \\
    y(t) = Cz(t), \quad t>0, \\
    z(0) =z_0 \in X.
\end{array}\right.
\end{equation}

In inverse problems framework the system \eqref{observability1} is referred to as the $direct \ problem$, which consists in determining the observation $y(t) = Cz(t)$ of the state $z(t)$ for given initial state $z_0$ and unbounded operator $A$. The $inverse \ problem$ aims to recover the initial state $z_0$ from the knowledge of the observation $y(t)$ over the time interval $t$ $\in$ $(0,T)$, where $T >0$ has to be chosen  sufficiently large.\\
\\

Inverse problems for evolution equations, motivated by a wide range of applications, have been a prominent focus of mathematical modeling and numerical research in recent decades \cite{IP, Yamamoto}. These problems are particularly challenging to solve due to their complex mathematical structure, their ill-posedness, and the typical limitation of having only partial data available. Many linear inverse problems that arise in evolution equations such as those in data assimilation, medical imaging, and geosciences, can be framed within the general framework of system \eqref{observability1} (see, for example, \cite{ACT, ACT2, ZY} and references therein).\\

The system \eqref{observability1} has a unique weak solution $z \in C(\mathbb{R},X)$ defined by 

\begin{equation}\label{solution1}
  z(t)= e^{itA} z_0
\end{equation}

If $z_0$ is not in $D(A)$, in general $z(t)$ does not belong to $D(A)$, and consequently  the second equation of system \eqref{observability1} might not be defined. We further make the following additional admissibility assumption on the observation operator $C$ to get this equation well defined for all $z_0$ in $X$ (see \cite{TW} and references therein).

\begin{definition}
    The operator $C$ in system \eqref{observability1} is  an {\sl admissible observation operator}  if for every $T>0$ there exists $C_{T}>0$ such that

    \begin{equation}\label{observability2}
        \int_0^T {\lVert C e^{itA} z_0 \rVert^{2}_{Y}} dt \leq C_{T} {\lVert z_{0} \rVert^{2}_{X}}, \quad 
 \forall z_{0} \in D(A).
\end{equation}
\end{definition}
Without loss of generality, we  assume that $C_T$ is a non decreasing function of $T$. Indeed  if the assumption is not fulfilled we substitute $C_{T} > 0$ by $ \mathrm{sup}_{0 \leq t \leq T}C_{T}$. In particular if $C \in \mathcal L(X, Y)$, one cane take $C_T =  T \|C\|$. Here $\mathcal L(X, Y)$ is the  set of all the bounded linear operators from $X$ to $Y$. Since $D(A)$ is dense in $X$, we deduce from the  admissibility condition \eqref{observability2} that $C$ can be uniquely extended as a bounded operator from $X$ to $L^2_{\textrm loc}(\mathbb R_+, Y)$, and $Cz(t)$ is well defined for almost everywhere on $\mathbb R_+$ for any $z_0 \in X$. \\

Let $\sigma(A_0) = (\lambda_{k})_{k \in \mathbb{N}^{*}}$ be the spectrum of $A_0^{\frac{1}{2}}$ where $(\lambda_{k})_{k \in \mathbb{N}^{*}}$ is a non-decreasing sequence of  positive real numbers. The orthonormal sequence of eigenvectors of $A_0^{\frac{1}{2}}$ in $H$ is denoted by
$(\phi_{k})_{k \in \mathbb{N}^{*}}$.\\

Let
\[
\mu_k=\operatorname{sign}(k)\lambda_{|k|}, \qquad
\psi_k=\frac{1}{\sqrt{2}}
\begin{pmatrix}
\dfrac{1}{i\mu_k}\phi_{|k|}\\[4mm]
\phi_{|k|}
\end{pmatrix},
\quad k\in\mathbb{Z}^*.
\]
By the definition of (A) in \eqref{operatorsAC}, its spectrum is given by
\[
\sigma(A)=\{\mu_k\}_{k\in\mathbb{Z}^*}.
\]
Moreover, the family $(\psi_k)_{k\in\mathbb{Z}^*}$ forms an orthonormal basis of eigenvectors of $A$ in $X$, with corresponding eigenvalues $(\mu_k)_{k\in\mathbb{Z}^*}$.

\begin{definition} \label{Afrequency}
    Let $z \in D(A) \setminus \{0\} \longrightarrow \lambda_A(z) \in \mathbb{R}$ be the $A$-frequency function defined by
    \begin{eqnarray}
  \lambda_A(z)&=& \langle Az , z \rangle_{X} \lVert z \rVert^{-2}_{X} \nonumber\\
 &=& \sum_{k\in \mathbb Z^*} \mu_{k}|\langle z , \psi_{k} \rangle_{X}|^{2} \left( \sum_{k\in \mathbb Z^*}  |\langle z , \psi_{k} \rangle_{X}|^{2}\right) ^{-1}.   \label{frequencyA}
\end{eqnarray}  
\end{definition}

\begin{remark}
    We note that the relation \begin{equation}         
    0 \leq \lVert (A -\lambda_A(z)I)z \rVert^{2}_{X} \lVert z \rVert^{-2}_{X} = \frac{\lVert Az \rVert^{2}_{X}}{\lVert z \rVert^{2}_X} - \lambda_A^{2}(z) 
    \end{equation}
    holds for all $z \in D(A)\setminus \{0\}$. In addition, the equality $\frac{\lVert Az \rVert^{2}_{X}}{\lVert z \rVert^{2}_{X}} - \lambda_A^{2}(z) = 0$ is satisfied if and only if $z = \psi_{k}$ for some $k$ $\in$ $\mathbb{Z}^{*}$.
 \end{remark}
Using the spectral decomposition of \(A\), we write
\[
A = A_+ - A_-,
\]
where
\[
A_\pm = \sum_{k=1}^{\infty} \lambda_k \phi_{\pm k}.
\]
We then define the unbounded operator
\[
|A| := A_+ + A_-.
\]

\begin{definition} \label{|A|frequency}
    Let $z \in D(A) \setminus \{0\} \longrightarrow \lambda_{|A|}(z) \in \mathbb{R}$ be the $|A|$-frequency function defined by
    \begin{eqnarray}
  \lambda_{|A|}(z)&=& \langle |A|z , z \rangle_{X} \lVert z \rVert^{-2}_{X} \nonumber\\
 &=& \sum_{k\in \mathbb Z^*} \lambda_{|k|}|\langle z , \psi_{k} \rangle_{X}|^{2} \left( \sum_{k\in \mathbb Z^*} | \langle z , \psi_{k} \rangle_{X}|^{2}\right) ^{-1}.   \label{frequency|A|}
\end{eqnarray}  
\end{definition}

We remark that we have  $|\lambda_{A}(z)| \leq \lambda_{|A|}(z)$ for all $z\in D(A) \setminus \{0\}$,  and  $|\lambda_{A}(z)| = \lambda_{|A|}(z)$  if $z$ is an eigenvector of $A$.

Let $\Sigma_d$ (resp. $\Sigma_i$) be the set of functions $\Psi : \mathbb{R} \longrightarrow \mathbb{R}_{+}^{*}$ even continuous, non-increasing (resp. non-decreasing). Recall that if $\Psi$ $\in$ $\Sigma_d$ is not bounded below by a positive constant it satisfies $\lim_{\lambda \to+ \infty} \Psi(\lambda) = 0$. Similarly if $\mathtt T  \in \Sigma_i$ is not upper bounded by a positive constant it verifies $\lim_{\lambda \to + \infty} \mathtt T(\lambda) = +\infty$.

\begin{definition}
The system \eqref{observability1} is said to be {\sl weakly observable} if there exists a pair of  functions $\mathtt T \in \Sigma_i$  and $\Psi \in \Sigma_d$   such that following observation inequality holds:

     \begin{equation}\label{observability3}
       \Psi\Big(\lambda_{|A|}(z_{0})\Big) {\lVert z_{0} \rVert^{2}_{X}} \leq    \int_0^{\mathtt T(\lambda_{|A|}(z_0))} {\lVert Ce^{itA}z_0\rVert^{2}_{Y}} dt ,\quad
 \forall z_{0} \in D(A).
    \end{equation}
    If $\mathtt T$ is upper  bounded and $\Psi$ is lower bounded simultaneously, the system is said to be {\sl exactly observable}.
    \end{definition}
    
    \begin{remark} The functions $\mathtt T$, $\Psi$ in the definition of the spectral coercivity of $C$ are not unique.
     Indeed the inequality \eqref{observability3} still holds for all pairs of functions $\left(\widetilde {\mathtt T}, \,
     \widetilde \Psi \right) \in \Sigma_i \times  \Sigma_d$ satisfying $ \widetilde \Psi (\lambda) \leq \Psi(\lambda)$ and  $\widetilde{\mathtt{T}}(\lambda) \geq \mathtt{T}(\lambda)$ for all $\lambda \geq \lambda_0$, where $\lambda_0$ is a fixed positive constant.
       
    \end{remark}

It is also known that observability and controllability are dual properties \cite{DR,Coron}. Most existing research on observability inequalities for systems of partial differential equations relies on time-domain methods, including nonharmonic series \cite{KL, AI, LT, FI}, the multiplier method \cite{Li}, and microlocal analysis techniques \cite{BLR, SU}. Alternatively, some other studies employ frequency-domain approaches, following the well-established Fattorini-Hautus criterion for finite-dimensional systems \cite{Fa, Ha,Tucsnak 2005, Miller, AT, ZY, Li, Zu}. \\

In \cite{AT} the following abstract infinite-dimensional dual observation system has been  considered:

\begin{equation}\label{observability}
\left\{ \begin{array}{lllcc}
   \dot {u} (t)  -i A_0 u(t) = 0 , \quad  t>0, \\
    w(t) = C_0u(t), \quad t>0, \\
    u(0) =z_0 \in X.
\end{array}\right.
\end{equation}

Precisely the authors  there have used frequency domain techniques, resolvent inequality and Fourier transform to 
derive a new frequency domain test for system \eqref{observability}'s observability that was formulated solely in terms of the operators $A_0$ and $C_0$. The time-domain system\eqref{observability} has been transformed into its frequency-domain equivalent, and the new found  test  primarily involves the solution in the frequency domain and the observability operator $C_0$. \\

Our objective here is to obtain a complete characterization for the
weak observability of the system \eqref{observability1} and equivalently the system \eqref{waveequationn} in the frequency domain in terms of the operators $A_0$ and $C_0$.  The frequency-domain test is particularly well-suited for numerical validation and calibration of physical models for several reasons: system parameters are typically measured in the frequency domain, and solving the system is generally more robust and efficient in this domain. In addition, in practice we only have access to the spectral properties of the operator $A_0$ and the action of the observation operator $C_0$ on the packets of eigenfunctions of $A_0$. \\

The paper is organized as follows. In Section~\ref{section2}, we provide a complete characterization of the weak observability of system~\eqref{observability1} in the frequency domain in terms of the operators $A$ and $C$. More precisely, Theorem~\ref{spectralcrvv1} establishes the equivalence between the spectral coercivity of $C$, introduced in Definition~\ref{spectralcoercivity}, and the weak observability inequality for system~\eqref{observability1}. The proofs of the results of this section are provided in the Appendix.\\

\noindent Section~\ref{section3}, which contains the main contribution of this paper, is devoted to the weak observability of the hyperbolic system~\eqref{waveequationn}. Our main result establishes a direct connection between resolvent estimates for the elliptic operator $A_0$ and the matrix-valued operator $A$. Specifically, Theorem~\ref{main1} shows that a resolvent estimate for $A_0$ implies a Hautus-type spectral observability estimate for $A$, which in turn yields the weak observability of the system (Corollary~\ref{maincorollary}).\\

\noindent Finally, in Section~\ref{section4}, we apply our abstract results to derive explicit weak boundary observability estimates for the wave equation on a rectangular domain under several geometric configurations of the observation region.

 \section{Characterization of weak Observability for the system  \eqref{observability1}}
\label{section2}
In this section we extend the results obtained in \cite{AT} to the system \eqref{observability1}. Precisely we
give a complete characterization for the weak observability of the system \eqref{observability1} in the frequency 
domain in terms of the operators $A$ and $C$. 

\begin{definition} \label{spectralcoercivity}
The operator $C$ is spectrally coercive if there exist functions $\varepsilon$, $\Psi$ $\in$ $\Sigma_d$ such that if $z \in$ $D(A) \setminus \{0\}$ satisfies
\begin{equation}
\frac{\lVert A z \rVert^{2}_{X}}{\lVert z \rVert^{2}_{X}} - \lambda_{A}^{2}(z) \leq \varepsilon(\lambda_{|A|}(z)),     
\end{equation}
then
\begin{equation}
    \lVert C z \rVert^{2}_{Y} \geq \Psi(\lambda_{|A|}(z)) \lVert z \rVert^{2}_{X}.
\end{equation}
\end{definition}

The following is a generalized Hautus-type test for the system \eqref{observability1}.

\begin{theorem}\label{spectralcrv}
    The operator $C$ is spectrally coercive if and only if there exist functions $\Psi$, $\varepsilon$ $\in \Sigma_d$ such that the following resolvent inequality holds:

    \begin{equation}\label{resolineqq1}
        \lVert z \rVert^{2}_{X} \leq \sup \Big\{ \frac{\lVert Cz \rVert^{2}_{Y}} {\Psi(\lambda_{|A|}(z))}  , \frac{\lVert (A -\lambda I)z \rVert^{2}_{X}}{(\lambda - \lambda_A(z))^{2} + \varepsilon(\lambda_{|A|}(z))} \Big \}, \quad \forall \lambda \in \mathbb{R}, \;\forall z \in D(A) \setminus  \{0\}.
    \end{equation}
    
\end{theorem}

The proof of the theorem is given in the Appendix.
\begin{remark}
Using the fact that $\frac{1}{2}(a+b) \leq \sup(a,b)\leq a+b$, for all positive reals $a$ and $b$, one can easily show that the inequality \eqref{resolineqq1} is equivalent to
\begin{equation}\label{resolineqq11}
        \lVert z \rVert^{2}_{X} \leq  \frac{\lVert Cz \rVert^{2}_{Y}} {\Psi(\lambda_{|A|}(z))}+ \frac{\lVert (A -\lambda I)z \rVert^{2}_{X}}{(\lambda - \lambda_A(z))^{2} + \varepsilon(\lambda_{|A|}(z))}, \quad \forall \lambda \in \mathbb{R}, \;\forall z \in D(A) \setminus  \{0\},
    \end{equation}
    for some $\Psi$, $\varepsilon$ $\in \Sigma_d$.
\end{remark}

 The following result provide an equivalence between the spectral coerciveness of $C$ and the weak observability
 inequality \eqref{observability3}.
 
\begin{theorem}\label{spectralcrvv1}
    \noindent The system \eqref{observability1} is weakly observable if and only if $C$ is spectrally coercive, that is, the following two assertions are equivalent.
    \\
    \\
    $(i)$  There exist functions $\varepsilon$, $\Psi$ $\in$ $\Sigma_d$ such that if $u$ $\in$ $D(A) \setminus \{0\}$ satisfying
\begin{equation}
\frac{\lVert A z \rVert^{2}_{X}}{\lVert z\rVert^{2}_{X}} - \lambda_{A}^{2}(z) \leq  \varepsilon(\lambda_{|A|}(z)),     
\end{equation}
then,
\begin{equation}
    \lVert C z \rVert^{2}_{Y} \geq \Psi(\lambda_{|A|}(z)) \lVert z \rVert^{2}_{X}.
\end{equation}
     \\
     $(ii)$ The following weak observability  inequality for the system \eqref{observability} holds:
     \begin{equation}\label{observability4}
      \theta_{2} \Psi \left (\theta_{0}\left(\frac{1}{T}+\lambda_{|A|}(z_{0})\right)\right) {\lVert z_{0} \rVert^{2}_{X}} \leq    \int_0^T {\lVert C z(t)\rVert^{2}_{Y}} dt,\quad
 \forall z_{0} \in D(A),
    \end{equation}
   \noindent for all $T \geq \mathtt T(\lambda_{|A|}(z_{0}))$, where   $\mathtt T(\lambda)$ is the unique solution $T$ for a given
   $\lambda$ to the equation
    \begin{equation} \label{Eq:T}
        T\varepsilon\left(\theta_{0} \left(\frac{1}{T} + \lambda\right)\right) = \theta_{1},
    \end{equation}

    \noindent and  $\varepsilon$, $\Psi$ $\in$ $\Sigma_d$ are the functions appearing in the spectral coercivity of $C$, and $\theta_{i}$, $i=0,1,2$ are positive universal constants that do not depend on $z_0$. In addition, $\mathtt T \in \Sigma_i$.
 \end{theorem}

\noindent The above theorem is somehow  an extension of several results in the literature \cite{BZ}, \cite{Ha},  \cite{Tucsnak 2005}, \cite{ZY}. We provide its proof in the Appendix. Since $\Psi$ is a non decreasing function
inequality \eqref{observability4} implies  the weak observability inequality \eqref{observability3}.   \\

\section{Weak observability for second order  infinite dimensional evolution systems} \label{section3}

In this section, we investigate an important particular case fitting in the framework of the previous section.
Recall that for $(u_0, u_1) \in X$, the second order evolution system \eqref{waveequationn} can be  transformed
into the first order evolution  system \eqref{observability1} by taking $z_0= (u_0, u_1)^t $, $z= (u, \dot u)^t$, and $y = w$. Our objective in this section is to derive a  test  for the
weak observability of the system \eqref{observability1} and equivalently the weak observability of the system \eqref{waveequationn} in the frequency domain in terms of the operators $A_0$ and $C_0$ that is  more explicit and practical  than the  one given in Theorem \ref{spectralcrvv1}.\\

\noindent Next we derive the  following  intermediate result
that is of interest itself.

\begin{theorem}\label{observabilitywavee}
Let $\delta_0 \in \Sigma_d$ and assume that the following inequality 
 \begin{equation}\label{stabestt1}   
 \delta_0(\omega) \Big (\omega^2  \lVert u \rVert^2_{H} + \lVert A_0^{\frac{1}{2}} u \rVert^2_{H}  \Big ) \leq  \lVert (A_0 - \omega^2I)u \rVert^2_{H} + \lVert C_0 u \rVert^2_{Y }, 
    \end{equation}
holds  for all $ \omega \in \mathbb{R}$ and for all $u \in H_2$.\\

 Then, there exists $\delta \in \Sigma_d$ such that the following inequality    
 \begin{eqnarray} \label{Hautus condition}
\delta(\lambda) \left \lVert z \right \rVert^{2}_{X}\leq \left \lVert (A - \lambda I)z \right \rVert^{2}_{X} + \left \lVert Cz \right \rVert^{2}_{Y},
 \end{eqnarray}
 is satisfied for all $\lambda \in \mathbb{R}$ and for all $z \in D(A)$.

\end{theorem}
\begin{proof}
\noindent 
Recall that   $(\phi_{k})_{k \in \mathbb{N}^{*}}$ is  the orthonormal sequence of eigenvectors of $A_0^{\frac{1}{2}}$ associated to the eigenvalues $(\lambda_{k})_{k \in \mathbb{N}^{*}}$. We first show the following optimal 
estimate.

\begin{lemma}\label{decompspc}
The following inequality 
    \begin{align} \label{interm7} 
    \rVert (A-\omega I)z \lVert^{2}_{X} \geq
  \sum_{k \in \mathbb{N^{*}}} (|\omega| - \lambda_k)^2 \lambda_k^2 |u_k|^2 +  \sum_{k \in \mathbb{N^{*}}} (|\omega| - \lambda_k)^2  |v_k|^2,
   \end{align}
   holds  for all $\omega \in \mathbb{R}$, and $z=(u,v)=\Big(\sum_{k \in \mathbb{N^{*}}} u_k \phi_k,\sum_{k \in \mathbb{N^{*}}} v_k \phi_k\Big) \in D(A)$.

\end{lemma}
\begin{proof}
 \noindent Let $z = (u,v) \in D(A)$, we have
\begin{align} \rVert (A-\omega I)z \lVert^{2}_{X} 
  &= \lVert (-\omega u -iv , iA_0 u - \omega v) \lVert^{2}_{H_1 \times H} \nonumber\\
&=  \lVert (A_0^{\frac{1}{2}}(\omega u + iv ) \rVert^{2}_{H} + \rVert iA_0 u - \omega v \lVert^{2}_{H}.   \label{inter1} 
  \end{align} 

Let
\[
u_r=\Re(u), \qquad v_r=\Re(v), \qquad u_i=\Im(u), \qquad v_i=\Im(v),
\]
and denote by $(u_{r,k}), (v_{r,k}), (u_{i,k}), $and $(v_{i,k})$ the coefficients in the respective spectral decompositions of $(u_r), (v_r), (u_i),$ and $(v_i).$\\

Hence  \eqref{inter1} is equivalent to the  the following inequality
\begin{align*} \rVert (A-\omega I)z \lVert^{2}_{X} 
&=   \lVert (A_0^{\frac{1}{2}}(\omega u + iv ) \rVert^{2}_{H} + \rVert iA_0 u + \omega v \lVert^{2}_{H} \\
&= \rVert A_0^{\frac{1}{2}} (v_i - wu_r) \lVert^{2}_{H} + \rVert -A_0 u_r + \omega v_i \lVert^{2}_{H} + \rVert A_0^{\frac{1}{2}} (v_r + wu_i) \lVert^{2}_{H} + \rVert A_0 u_i + \omega v_r \lVert^{2}_{H}.
  \end{align*} 
\noindent Then, applying the spectral decomposition of $u$ and
$v$, along with Young's inequality, we obtain:

\[
\begin{aligned}
&\left\| A_0^{1/2}(v_i-\omega u_r)\right\|_H^2
 + \left\| -A_0u_r+\omega v_i\right\|_H^2 \\
&\quad
 = \sum_{k\in\mathbb N^*}
 \Bigl[
   \lambda_k^2 (v_{i,k}-\omega u_{r,k})^2
   + (\omega v_{i,k}-\lambda_k^2u_{r,k})^2
 \Bigr] \\
&\quad
 = \omega^2 \sum_{k\in\mathbb N^*}
   \bigl(\lambda_k^2u_{r,k}^2+v_{i,k}^2\bigr)
   -4\omega \sum_{k\in\mathbb N^*}
   \lambda_k^2u_{r,k}v_{i,k} \\
&\qquad
   + \sum_{k\in\mathbb N^*}
   \lambda_k^2\bigl(\lambda_k^2u_{r,k}^2+v_{i,k}^2\bigr) \\
&\quad
 \ge
 \omega^2 \sum_{k\in\mathbb N^*}\lambda_k^2u_{r,k}^2
 -2|\omega|\sum_{k\in\mathbb N^*}\lambda_k^3u_{r,k}^2
 +\sum_{k\in\mathbb N^*}\lambda_k^4u_{r,k}^2 \\
&\qquad
 +\omega^2\sum_{k\in\mathbb N^*}v_{i,k}^2
 -2|\omega|\sum_{k\in\mathbb N^*}\lambda_k v_{i,k}^2
 +\sum_{k\in\mathbb N^*}\lambda_k^2v_{i,k}^2 \\
&\quad
 = \sum_{k\in\mathbb N^*}
   \lambda_k^2u_{r,k}^2
   \bigl(\omega^2-2|\omega|\lambda_k+\lambda_k^2\bigr) \\
&\qquad
   + \sum_{k\in\mathbb N^*}
   v_{i,k}^2
   \bigl(\omega^2-2|\omega|\lambda_k+\lambda_k^2\bigr) \\
&\quad
 = \sum_{k\in\mathbb N^*}
   (|\omega|-\lambda_k)^2\lambda_k^2u_{r,k}^2
   + \sum_{k\in\mathbb N^*}
   (|\omega|-\lambda_k)^2v_{i,k}^2 .
\end{aligned}
\]


Following the same steps, we similarly obtain
\begin{align*}
\rVert A_0^{\frac{1}{2}} (v_r + wu_i) \lVert^{2}_{H} + \rVert A_0 u_i + \omega v_r \lVert^{2}_{H} \geq 
\sum_{k \in \mathbb{N^{*}}} (|\omega| - \lambda_k)^2 \lambda_k^2 u_{i,k}^2 +  \sum_{k \in \mathbb{N^{*}}} (|\omega| - \lambda_k)^2  v_{r,k}^2.
\end{align*}

Adding  the last two inequalities provide the estimate \eqref{interm7}.

\end{proof}

\noindent Back to the proof of the theorem, for a fixed $\omega \in \mathbb R$,   $u  =\sum_{k \in \mathbb{N^{*}}} u_k \phi_k$ has the following orthogonal
decomposition 
\begin{equation} \label{interm2} u = u^0 + \tilde{u}, \end{equation}
 where 
\begin{equation} \label{interm3}
 u^0 =\sum_{\lambda_k \geq \frac{|\omega|}{3} }u_k \phi_k,  \;\;\; \tilde{u}  =\sum_{\lambda_k < \frac{|\omega|}{3} }u_k \phi_k.
\end{equation}
Indeed $u^0$ represents the high frequency part of $u$ while $\tilde u$ form the low frequency one.
\begin{lemma}\label{kk}
For any $u= \sum_{k \in \mathbb{N^{*}}} u_k \phi_k \in D(A_0)$ and for all $\omega \in \mathbb{R}$, the following inequality holds
        \begin{align} \label{interm6}
    \sum_{k \in \mathbb{N^{*}}} (|\omega| - \lambda_k)^2 \lambda_k^2 |u_k|^2 \geq \frac{1}{16} \lVert (A_0 -  \omega^2)u^0 \rVert^2_{H} + \frac{1}{9}\omega^2\lVert A_0^{\frac{1}{2}} \tilde u \rVert^2_{H}+3 \lVert A_0 \tilde{u} \rVert^2_{H}.
    \end{align}
   
\end{lemma}
\begin{proof}
Let $u \in D(A_0)$ and $\omega \in \mathbb{R}$. Therefore
\begin{align*}
    \sum_{\lambda_k \geq \frac{|\omega|}{3}} (|\omega| - \lambda_k)^2 \lambda_k^2 |u_k|^2 & \geq \frac{1}{16} \sum_{\lambda_k \geq \frac{|\omega|}{3}} (|\omega|^2 - \lambda^2_k)^2 |u_k|^2 \\
    & = \frac{1}{16} \lVert (A_0 - \omega^2)u^0 \rVert^2_{H},
\end{align*}
and
\begin{align*}
    \sum_{\lambda_k < \frac{|\omega|}{3}} (|\omega| - \lambda_k)^2 \lambda_k^2 |u_k|^2 & \geq \frac{4\omega^2}{9} \sum_{\lambda_k < \frac{|\omega|}{3}} \lambda^2_k |u_k|^2 \\
    & \geq \frac{\omega^2}{9} \sum_{\lambda_k < \frac{|\omega|}{3}} \lambda^2_k |u_k|^2 +3  \sum_{\lambda_k < \frac{|\omega|}{3}} \lambda^4_k |u_k|^2  = \frac{1}{9}\omega^2\lVert A_0^{\frac{1}{2}} \tilde u \rVert^2_{H}+3 \lVert A_0 \tilde{u} \rVert^2_{H}.
\end{align*}
\end{proof}
\begin{lemma}\label{kkk}
There exists $ \tilde c_0 = c_0(\|C_0\|) \in (0,1)$ such that the following inequality holds
    \begin{align} \label{interm5}
  \lVert C_0u \rVert^2_{Y} \geq \tilde c_0 \lVert C_0u^0 \rVert^2_{Y} - 2   \lVert A_0 \tilde{u} \rVert^2_{H},
  \end{align}
  for all  $u \in D(A_0)$ with $u_0$ and $\tilde u$ are defined in \eqref{interm2}-\eqref{interm3}.

\end{lemma}
\begin{proof}
Using the fact that $Y$ is a Hilbert space, we have 
\begin{align*}
   \lVert C_0u \rVert^2_{Y} = 
    \lVert C_0(u^0 + \tilde{u})\rVert^2_{Y} 
    = \lVert C_0u^0 \rVert^2_{Y} + \lVert C_0\tilde{u}\rVert^2_{Y} + 2\Re\left(\langle C_0u^0, C_0 \tilde{u} \rangle_Y\right).
\end{align*}
\\
Let  $  \varepsilon \in (0, 1)$ be fixed.  Using the Young-inequality, we obtain
\begin{align*}
    \lVert C_0u^0\rVert^2_{Y} + \lVert C_0\tilde{u}\rVert^2_{Y} + 2 \Re\left( \langle C_0u^0, C_0\tilde{u} \rangle_Y\right)
    & \geq \lVert C_0u^0\rVert^2_{Y} + \lVert C_0\tilde{u} \rVert^2_{Y} - \varepsilon^2 \lVert C_0u^0\rVert^2_{Y} - \frac{1}{\varepsilon^2} \lVert C_0\tilde{u}\rVert^2_{Y} \\ 
    & = (1- \varepsilon^2) \lVert C_0u^0\rVert^2_{Y} -(\frac{1}{\varepsilon^2}-1) \lVert C_0\tilde{u}\rVert^2_{Y}.
\end{align*}

Since  $C_0$ is a linear bounded operator from $D(A_0)$ to $Y$, we have 
   \begin{align*}
    \lVert C_0\tilde{u}\rVert^2_{Y} \leq \|C_0\|^2\lVert \tilde{u} \rVert^2_{H_2} = \|C_0\|^2 \lVert A_0 \tilde{u} \rVert^2_{H}.
\end{align*}

Consequently 
\begin{align*}
    \lVert C_0u\rVert^2_{Y} \geq (1- \varepsilon^2) \lVert C_0u^0\rVert^2_{Y} -(\frac{1}{\varepsilon^2}-1) \|C_0\|^2 \lVert A_0 \tilde{u} \rVert^2_{H}.
    \end{align*}

Now taking $\varepsilon = (1+2\|C_0\|^{-2})^{-\frac{1}{2}}$, we obtain the desired inequality, with $\tilde c_0 = 2(2+\|C_0\|^{-2})^{-1}$.
\\
\end{proof}

\noindent Combining  inequalities \eqref{interm6}, \eqref{interm5},  we obtain

    \begin{align}
\sum_{k \in \mathbb{N}^{*}}
\bigl(|\omega|-\lambda_k\bigr)^2 \lambda_k^2 |u_k|^2
+\|C_0u\|_{Y}^{2}
&\geq
\frac{1}{16}\|(A_0-\omega^2 I)u^0\|_{H}^{2}
+\tilde c_0 \|C_0u^0\|_{Y}^{2}
+\frac{1}{9}\omega^2\|A_0^{1/2}\tilde u\|_{H}^{2}
+\|A_0\tilde u\|_{H}^{2}
\nonumber\\
&\geq
\min\!\left(\frac{1}{16},\tilde c_0\right)
\left(
\|(A_0-\omega^2 I)u^0\|_{H}^{2}
+\|C_0u^0\|_{Y}^{2}
\right)
\nonumber\\
&\qquad
+\frac{1}{9}
\left(
\|A_0^{1/2}\tilde u\|_{H}^{2}
+\omega^2\|\tilde u\|_{H}^{2}
\right).
\label{interm8}
\end{align}

Applying  the inequality \eqref{stabestt1}  to $u_0$, we get 

\begin{equation} 
 \delta_0(\omega) \Big (\lVert A_0^{\frac{1}{2}} u_0 \rVert^2_{H} + \omega^2  \lVert u_0 \rVert^2_{H} \Big ) \leq  \lVert (A_0 - \omega^2I)u_0 \rVert^2_{H} + \lVert C_0 u_0 \rVert^2_{Y }. \label{interm70}
    \end{equation}

Using inequalities \eqref{interm70} and \eqref{interm8}, we find  
\begin{align}
\sum_{k \in \mathbb{N}^{*}}
\bigl(|\omega|-\lambda_k\bigr)^2 \lambda_k^2 |u_k|^2
+\|C_0u\|_{Y}^{2}
&\geq
\min\!\left(\frac{1}{16},\tilde c_0\right)
\delta_0(\omega)
\left(
\|A_0^{1/2}u^0\|_{H}^{2}
+\omega^2\|u^0\|_{H}^{2}
\right)
\nonumber\\
&\qquad
+\frac{1}{9}
\left(
\|A_0^{1/2}\tilde u\|_{H}^{2}
+\omega^2\|\tilde u\|_{H}^{2}
\right)
\nonumber\\
&\geq
\min\!\left(
\min\!\left(\frac{1}{16},\tilde c_0\right)\delta_0(\omega),
\frac{1}{9}
\right)
\left(
\|A_0^{1/2}(u^0+\tilde u)\|_{H}^{2}
+\omega^2\|u^0+\tilde u\|_{H}^{2}
\right)
\nonumber\\
&\geq
\tilde{\delta}(\omega)
\left(
\|A_0^{1/2}u\|_{H}^{2}
+\omega^2\|u\|_{H}^{2}
\right),
\label{interm11}
\end{align}
where $\tilde \delta(\omega) = \min\left(\min (\frac{1}{16}, \tilde c_0)\delta_0(\omega), \frac{1}{9}\right)$. Since $\delta_0\in \Sigma_d$,
we have $\tilde \delta \in \Sigma_d$ as well.\\

\noindent Now, we distinguish  two cases:\\

 \noindent  {\it Case 1.} {$\frac{1}{2} \lVert v \rVert^2_{H} <  \lVert  \omega u \rVert^2_{H}$}.\\

Inequalities \eqref{interm7} and  \eqref{interm11},  then lead to
\begin{align*}
    \left \lVert (A - wI)z \right \rVert^{2}_{X} + \left \lVert Cz \right \rVert^{2}_{Y} \geq 
    \frac{1}{3} \tilde \delta(\omega)  \Big( \lVert A_0^{\frac{1}{2}} u \rVert^2_{H}  + \lVert v \rVert^2_{H} \Big ) =\delta(\omega) \lVert z \rVert^2_{X},
\end{align*} 
where $\delta = \frac{1}{3}\tilde \delta$.

\noindent  {\it Case 2: } {$\frac{1}{2} \lVert v \rVert^2_{H} \geq  \lVert \omega u \rVert^2_{H}$}. \\

 We first  deduce the following inequality 
 \begin{equation} \label{interm12}
 \lVert v -\omega u \rVert^2_{H} \geq \frac{1}{2} \lVert v \rVert^2_{H}.
 \end{equation}

 Thus, we obtain:
\begin{align*} \rVert (A-\omega I)z \lVert^{2}_{X} 
&\geq \rVert A_0^{\frac{1}{2}} (v_i - wu_r) \lVert^{2}_{H} + \rVert A_0 u_r - \omega v_i \lVert^{2}_{H} \\
& \geq \rVert A_0^{\frac{1}{2}} (v - wu) \lVert^{2}_{H} \\
& \geq \lambda_1^2 \rVert  v - wu \lVert^{2}_{H},
\end{align*}
which combined with \eqref{interm12}, give
\begin{align*} \rVert (A-\omega I)z \lVert^{2}_{X} \geq \frac{\lambda_1^2}{2} \lVert v \rVert^2_{H}.
\end{align*}

Hence 
\begin{align*}
    \left \lVert (A - \omega I)z \right \rVert^{2}_{X} + \left \lVert Cz \right \rVert^{2}_{Y} \geq  \frac{\tilde \delta(\omega)}{2}  \Big( \lVert A_0^{\frac{1}{2}} u \rVert^2_{H} + \omega^2  \lVert u \rVert^2_{H} \Big ) + \frac{\lambda_1^2}{4}\lVert v \rVert^2_{H}
\end{align*}
\\
\\
From the previous inequalities, we deduce that for all $\omega \in \mathbb{R}$, we have  
\begin{align*}
    \left \lVert (A - \omega I)z \right \rVert^{2}_{X} + \left \lVert Cz \right \rVert^{2}_{Y} \geq  \delta(\omega)  \Big( \lVert A_0^{\frac{1}{2}} u \rVert^2_{H}  +\lVert v \rVert^2_{H} \Big ),
\end{align*}
where $\delta(\omega) =\min(\frac{\tilde \delta(\omega)}{2},\frac{\lambda_1^2}{4})$. It is forward to verify that 
$\delta$ lies  indeed in $\Sigma_d$.\\

In both cases, we proved the existence of a positive  function $\delta \in \Sigma_d$ such that the following inequality holds
\begin{align*} 
\delta(\omega) \left \lVert z \right \rVert^{2}_{X} \leq \left \lVert (A - wI)z \right \rVert^{2}_{X} + \left \lVert Cz \right \rVert^{2}_{Y},
 \end{align*}
for all $\omega \in \mathbb R$ and $z\in D(A)$, which finishes the proof of Theorem \ref{observabilitywavee}.

\end{proof}

\noindent The following is the main theorem of this section.
\begin{theorem}\label{main1}
  Assume that there exist functions $\Psi_0$, $\varepsilon_0$ $\in \Sigma_d$ such that the following resolvent inequality holds:

    \begin{equation}\label{resolineq1}
        \omega^2\lVert u \rVert^{2}_{H} + \lVert A_0^{\frac{1}{2}}u \rVert_H^2 \leq \sup \Big\{ \frac{\lVert C_0u \rVert^{2}_{Y}} {\Psi_0(\omega)}  , \frac{\lVert (A_0 -\omega^2I)u \rVert^{2}_{H}}{ \varepsilon_0(\omega)} \Big \}, \quad \forall \omega \in \mathbb{R}, \;\forall u \in D(A_0) \setminus \{0\},
    \end{equation} 

Then, the system \eqref{observability1} is weakly observable. That is, there exist functions $\Psi$, $\varepsilon$ $\in \Sigma_d$ such that the following resolvent inequality holds:

    \begin{equation}\label{resolineq}
        \lVert z \rVert^{2}_{X} \leq \sup \Big\{ \frac{\lVert Cz \rVert^{2}_{Y}} {\Psi(\lambda_{|A|}(z))}  , \frac{\lVert (A -\lambda I)z \rVert^{2}_{X}}{(\lambda - \lambda_A(z))^{2} + \varepsilon(\lambda_{|A|}(z))} \Big \},  \quad  \forall \lambda \in \mathbb{R}, \; \forall z \in D(A) \setminus  \{0\}.
    \end{equation}
where  $\lambda_{A}(z)$  (resp. $\lambda_{|A|}(z)$) is  the $A$-frequency  (resp. $|A|$-frequency)of  $z$ defined in  Definition \ref{Afrequency} (resp. \eqref{Afrequency}).
\end{theorem}

\begin{proof}

\noindent Recall that  $X = H_1 \times H$ is  a Hilbert space equipped with the following inner product:

$$
 \Big\langle (u,v), (u_1,v_1) \Big\rangle_X = \langle u,u_1\rangle_{H_1} + \langle v,v_1\rangle_{H} = \langle A_0^{\frac{1}{2}}u,A_0^{\frac{1}{2}}u_1\rangle_{H} + \langle v,v_1\rangle_{H}. 
$$

The inequality \eqref{resolineq1}    implies 
   
\begin{equation} \min{( \varepsilon_0(\omega),\Psi_0(\omega))} \left(\omega^2\lVert u \rVert^{2}_{H} + \lVert A_0^{\frac{1}{2}}u \rVert_H^2 \right) \leq \lVert (A_0 -\omega^2I)u \rVert^{2}_{H} + \lVert C_0u \rVert^{2}_{Y}.
    \end{equation}

\noindent Denote $\delta_0(\omega) =  \min{\Big( \varepsilon_0(\omega),\Psi_0(\omega)\Big)}$, which by the way belongs to  $\Sigma_d$ since both $\varepsilon_0$, $\Psi_0$  lie in $ \Sigma_d$. \\

Back to the proof of Theorem \ref{main1}. Inequality \eqref{resolineq1}  implies 

\begin{equation}\label{stabestt}    
\delta_0(\omega) \Big (\omega^2  \lVert u \rVert^2_{H} + \lVert A_0^{\frac{1}{2}} u \rVert^2_{H}  \Big ) \leq  \lVert (A_0 - \omega^2I)u \rVert^2_{H} + \lVert C_0 u \rVert^2_{Y }, \;\; \forall \omega \in \mathbb R, \; \forall u\in H_2, \end{equation}
    where $\delta_0^{-1}(\omega) = \sup\left(\frac{1}{\Psi_0(\omega)}, \frac{1}{\varepsilon_0(\omega)} \right)$. Since
    $\Psi_0$ and $\varepsilon_0$ belong to $\Sigma_d$, we have $\delta_0 \in \Sigma_d$. \\
    
Now using the results of Theorem \ref{observabilitywavee}, we deduce the existence of $\delta\in \Sigma_d$ such that 

 \begin{align}  \label{interm13}
\delta(\lambda) \left \lVert z \right \rVert^{2}_{X} \leq \left \lVert (A - \lambda I)z \right \rVert^{2}_{X} + \left \lVert Cz \right \rVert^{2}_{Y},
 \end{align}
for all $\lambda \in \mathbb R$ and   $z\in D(A)$. \\
    
Next,  we shall show that  the operator $C$ is spectrally coercive. \\

 Let  $z\in D(A)\setminus\{0\}$,  $\widetilde \varepsilon(\lambda)=\frac{\delta(\lambda)}{4}$, $\widetilde \Psi(\lambda)=\frac{\delta(\lambda)}{4}$, and assume 

\begin{equation}\label{interm14}
\lVert (A - \lambda_A(z) I)z \rVert^{2}_{X}  \leq \widetilde \varepsilon(\lambda_{|A|}(z))\lVert z \rVert^{2}_{X}.
\end{equation}

Combining \eqref{interm13} with $\lambda = \lambda_A(z)$, and \eqref{interm14}, we deduce 

\begin{equation*}
    \lVert Cz \rVert^{2}_{Y} \geq \widetilde \Psi(\lambda_A(z)) \lVert z \rVert^{2}_{X}.
\end{equation*}
Since $|\lambda_A(z)| \leq \lambda_{|A|}(z)$, and $\Psi$ is even, and  decreasing  on $\mathbb R_+$, we also have 
\begin{equation*}
    \lVert Cz \rVert^{2}_{Y} \geq \widetilde \Psi(\lambda_{|A|}(z)) \lVert z \rVert^{2}_{X}.
\end{equation*}
By construction $\widetilde \varepsilon, \widetilde \Psi \in \Sigma_d$, and by consequence  the previous inequality 
shows that $C$ is spectrally coercive.\\

\noindent By Theorem \ref{spectralcrv}, there exist functions $\Psi$, $\varepsilon$ $\in \Sigma_d$ such that for all $\omega \in \mathbb{R}$ and for all $z \in D(A) \setminus  \{0\}$ the following resolvent inequality holds

    \begin{equation*}
        \lVert z \rVert^{2}_{X} \leq \sup \Big\{ \frac{\lVert Cz \rVert^{2}_{Y}} {\Psi(\lambda_{|A|}(z))}  , \frac{\lVert (A -\lambda I)z \rVert^{2}_{X}}{(\lambda - \lambda_A(z))^{2} + \varepsilon( \lambda_{|A|}(z))} \Big \}, \quad \forall \lambda \in \mathbb R.
    \end{equation*}


\end{proof}

\begin{corollary} \label{maincorollary}
  Assume that there exist functions $\Psi_0$, $\varepsilon_0$ $\in \Sigma_d$ such that the following resolvent inequality holds:

    \begin{equation*}
        \omega^2\lVert u \rVert^{2}_{H} + \lVert A_0^{\frac{1}{2}}u \rVert_H^2 \leq \sup \Big\{ \frac{\lVert C_0u \rVert^{2}_{Y}} {\Psi_0(\omega)}  , \frac{\lVert (A_0 -\omega^2I)u \rVert^{2}_{H}}{ \varepsilon_0(\omega)} \Big \}, \quad \forall \omega \in \mathbb{R}, \;\forall u \in D(A_0) \setminus \{0\},
    \end{equation*}

 Then,  there exists a pair of  functions $\mathtt T \in \Sigma_i$  and $\Psi \in \Sigma_d$   such that following weak
 observation inequality for the system  \eqref{observability1} holds:

     \begin{equation*}
       \Psi\Big(\lambda_{|A|}\left((u_0,u_1)^t\right)\Big) \left(\|A_0^{\frac{1}{2}}u_0\|_H^2+\|u_1\|_H^2 \right) \leq    \int_0^{\mathtt T\left(\lambda_{|A|}\left((u_0,u_1)^t\right) \right)} {\lVert C_0u(t)\rVert^{2}_{Y}} dt ,\quad
 \forall (u_0, u_1) \in H_2 \times H_1,
    \end{equation*}
where $\lambda_{|A|}\left( (u_0,u_1)^t \right)=  \left(\|A_0 u_0\|_H^2+\|A_0^{\frac{1}{2}}u_1\|_H^2\right) \left(\|A_0^{\frac{1}{2}}u_0\|_H^2+\|u_1\|_H^2\right)^{-1}.$
\end{corollary}

\begin{proof}
The results of the corollary is a direct consequence of Theorems \ref{spectralcrvv1} and  \ref{main1}. 
\end{proof}

\begin{remark}
The functions $\Psi$ and $\mathtt T$ appearing in Corollary \ref{maincorollary} can be determined explicitly in terms of the functions $\Psi_0$ and $\varepsilon_0$ by tracing the proof of Theorem \ref{main1}. In the next section, we provide several examples illustrating such constructions.

\end{remark}

\section{Boundary observability of the wave equation in a rectangle} \label{section4}

Let $\Omega = (0,a\pi) \times  (0,b\pi)$ where $a, b>0$ are  given constants. We denote by $\partial\Omega$ the boundary of $\Omega$, and by $n$ the outward unit normal vector on $\partial\Omega$. We consider the following initial-boundary value problem:

\begin{equation}\label{waveequation}
\left\{
\begin{aligned}
    &\ddot{u}(x,t) + A_0 u(x,t) = 0
    && x\in\Omega,\quad t>0, \\
    &u(x,t)=0
    && \text{on } \partial\Omega \times \mathbb{R}_{+}, \\
    &u(x,0)=u_0(x), \qquad
    \dot{u}(x,0)=v_0(x),
\end{aligned}
\right.
\end{equation}
where $A_0= -\Delta $ denote  the Dirichlet Laplacian on the Hilbert space $H=L^2(\Omega)$. Its associated scale of spaces is given by
\[
H_1=H_0^1(\Omega), \qquad
H_2=D(A_0)=H^2(\Omega)\cap H_0^1(\Omega).
\]
Let $\Gamma$ be a nonempty open subset of $\partial\Omega$, and set
\[
Y=L^2(\Gamma).
\]
We define the observation operator $C_0:\;  H\to Y$ by
\[
C_0u=\frac{\partial u}{\partial n}\Big|_{\Gamma},
\]
where $\partial u/\partial n$ denotes the outward normal derivative of $u$ on $\Gamma$. One can verify that the  linear operator  $C_0 $ is bounded from $D(A_0)$  onto $L^2(\Gamma)$. \\

\noindent We consider here the inverse problem of recovering the initial state $(u_0,v_0) \in X= H_0^1(\Omega) \times L^2(\Omega)$ of the wave equation \eqref{waveequation}  from the observation $C_0 u(t) $ for $t\in (0, T)$ with $T>0$ large enough.  The wave equation \eqref{waveequation} can be reformulated in the first evolution system \eqref{observability1}. Here  $X = H_0^1(\Omega) \times L^2(\Omega$) is the  Hilbert space with scalar product:
$$
\langle z, z' \rangle_{H_0^1(\Omega) \times L^2(\Omega)} = \Big\langle (u,v), (u_1,v_1) \Big\rangle_{H_0^1(\Omega) \times L^2(\Omega)}  = \langle \nabla u,\nabla u_1\rangle_{L^2(\Omega)} + \langle v,v_1\rangle_{L^2(\Omega)}. 
$$

\noindent  Since $A_0 = -\Delta$, we have $A = -i \begin{pmatrix} 0 & I \\ \Delta & 0 \end{pmatrix}$ : D($A$) $\subset$ $X$ $\rightarrow$ $X$ is a linear bounded self-adjoint operator with a compact resolvent. Hence, the operator $iA$ generates a strongly continuous group of isometries in $X$ denoted by $(e^{itA})_{t {\in \mathbb{R}} }$ \cite{TW}.

\begin{equation}\label{obsrvoperator}
y(t)= Cz(t) = C\begin{pmatrix} u(t) \\ v(t)\end{pmatrix}(t) =(C_0  \ 0)\begin{pmatrix} u(t) \\ v(t)\end{pmatrix} \;= \;\frac {\partial u }{\partial n }\Big |_\Gamma. 
\end{equation}

Since $ D(A) = H^2(\Omega) \cap H_0^1(\Omega) \times H_0^1(\Omega)$, the observability operator $C : D(A) \to Y := L^2(\Gamma)$, defined by $\eqref{obsrvoperator}$, is a bounded operator. In addition, $C$ is an admissible observability operator, that is, for any $T>0$ there exists a constant $C_T >0$, such that the following inequality holds
\cite{Li}:

\begin{equation}\label{observabilityy}
        \int_0^T \int_\Gamma  \Big |\frac{\partial u}{\partial n}      \Big |^{2} dx dt \leq C_{T} \int_\Omega (| \nabla u_0 | ^2 + | v_0 |^2) dx, \quad 
 \forall z_{0}=(u_0,v_0) \in D(A).
\end{equation}

\noindent Next, we derive observability inequalities corresponding to different choices of  the observability set $\Gamma$.
\\

\textbf{Assumption I :} We assume that $\Gamma$ contains at least two touching sides of $\Omega$.\\

Consider the Helmholtz equation defined by

\begin{equation}\label{helmotzequationn}
\left\{ \begin{array}{lllcc}
    \Delta \phi + \omega^2 \phi  = f , \quad x \in \Omega, \\
    \phi = 0 \quad x \in \ \partial \Omega \setminus \overline{\Gamma}, \\
    \frac{\partial \phi }{\partial n}  - i \omega \phi = g \quad x \in \Gamma,
\end{array}\right.
\end{equation}
where $g \in L^{2}(\Gamma)$ and $f \in L^{2}(\Omega)$.

\begin{theorem}
    The system \eqref{waveequation} is exactly observable meaning that there exist  constants $ c = c(\Gamma, \|C_0\|)>0$ and $T_0= T_0(\Gamma, \|C_0\|) > 0$ such that the  inequality
     \begin{equation} \label{obscarre}
   c \Big ({\lVert\nabla u_{0} \rVert^{2}_{L^2(\Omega)}} + \lVert v_0 \rVert^{2}_{L^2(\Omega)}\Big) \leq    \int_0^{T} \int_{\Gamma} \Big| \frac{\partial u}{\partial n} \Big|^2 ds(x) dt, 
    \end{equation}
holds  for all $z_{0}=(u_0,v_0) \in  H^2(\Omega) \cap H_0^1(\Omega) \times H_0^1(\Omega)$ and for all $T \geq T_0$. \end{theorem}
\begin{proof}
The following result has been shown using Rellich identities (which are somehow related to the multiplier approach in observability inequalities \cite{Li}) \cite{UH}.
\begin{lemma}  Under Assumption I on $\Gamma$, a solution $u \in H^{1}(\Omega)$ to the system \eqref{helmotzequationn} satisfies the following estimate:
\begin{equation}\label{stabest}  c_1 \Big (  \lVert \nabla u \rVert^2_{L^2(\Omega)} + \omega^2  \lVert u \rVert^2_{L^2(\Omega)} \Big ) 
 \leq  \lVert (\Delta + \omega^2)u \rVert^2_{L^2(\Omega)} + \Big\lVert \frac{\partial u}{\partial n} \Big\rVert^2_{L^2(\Gamma) },     \end{equation}
     for all $ \omega \in \mathbb{R},$ where $c_1 > 0$ is a positive constant that only depends on $\Gamma$.
\end{lemma}
Following the proof of  Theorem \ref{observabilitywavee} there exists a constant $c_2= c_2(\Gamma, \|C_0\|)>0$ such that the following inequality:

 \begin{eqnarray} \label{Hautus condition}
 c_2 \left \lVert z \right \rVert^{2}_{H_0^1(\Omega) \times L^2(\Omega)} \leq 
\left \lVert (A - \lambda I)z \right \rVert^{2}_{H_0^1(\Omega) \times L^2(\Omega)} + \left \lVert Cz \right \rVert^{2}_{L^{2}(\Gamma)},\end{eqnarray}
 holds for all  $\lambda \in \mathbb{R}, $ and all $   z \in D(A)$.\\

\noindent
It follows from the inequality \eqref{Hautus condition} that the operator $C$ is spectrally coercive with
\[
\Psi(\lambda)=\varepsilon(\lambda)=\frac{c_2}{2}, \quad \forall \lambda \in \mathbb R.
\]
Therefore, by Theorem~\ref{spectralcrvv1}, the observability inequality \eqref{obscarre} holds with
\[
T_0=\frac{2\theta_1}{c_2}
\quad\text{and}\quad
c=\frac{\theta_2 c_2}{2}.
\]

\end{proof}

\begin{remark}
Since $\Gamma$ contains at least two adjacent sides of $\Omega$, it satisfies the geometric assumption of \cite{BLR}. Consequently, the system is exactly controllable.\

\noindent
Furthermore, the geometric control condition is not only sufficient but also necessary for the exact controllability of the wave equation, as proved in \cite{BLR} (see also \cite{BG} for the case of boundary control).

\end{remark}

\noindent \textbf{Assumption II :} We assume that $\Gamma$ is  a one side  of $\Omega$. Without loss of generality, we further assume that $\Gamma = (0,a\pi) \times \{ 0 \}$. \\

\noindent Since the geometric assumption of \cite{BLR} is not satisfied in the present setting, we establish a new weak observability result for the wave equation with observations restricted to one side of the rectangular domain~$\Omega$.

\begin{theorem} The system \eqref{waveequation} is weakly observable, that is,  the solution $u$ satisfies 
the following 
  \begin{equation}
   \Psi\Big(\lambda_{|A|}\left((u_0,u_1)^t\right)\Big) \Big (\lVert \nabla u_{0} \rVert^{2}_{L^2(\Omega)} + \lVert u_1\rVert^{2}_{L^2(\Omega)}\Big ) \leq    \int_0^{T} \int_{\Gamma} \Big| \frac{\partial u}{\partial n} \Big|^2 ds(x) dt, \quad
    \end{equation}
 for all $z_{0}=(u_0, u_1) \in  H^2(\Omega) \cap H_0^1(\Omega) \times H_0^1(\Omega)$ and for all 
 $$T \geq \mathtt T(\lambda_{|A|}\left((u_0,u_1)^t\right)), $$
 where  $\Psi \in \Sigma_d$   and  $ \mathtt T \in \Sigma_i$  are respectively given by 
\begin{equation}
    \Psi(\lambda) = \frac{c_4}{c_3+c_2\lambda^2}, \qquad   \mathtt T(\lambda)=\Big(c_1 + c_2\lambda^2\Big),
 \end{equation}
 and 
 \begin{equation}
 \lambda_{|A|}\left((u_0,u_1)^t\right) = \big(\|\Delta u_0\|_{L^2(\Omega)}^2 +\|\nabla u_1\|_{L^2(\Omega)}^2 \big)
 \left( \|\nabla u_0\|_{L^2(\Omega)}^2 + \|u_1\|^2_{L^2(\Omega)}\right)^{-1},
 \end{equation}
 with  $c_i = c_i(\Gamma, \Omega, \|C_0\|)>0, \; i=1, \,3,\,4,$ and $c_2>0$ is an universal constant.
\end{theorem}
\begin{proof}
 This following result has been derived  in \cite{BY}.
 \begin{lemma} Under Assumption II on $\Gamma$, a solution $u \in H^{1}(\Omega)$ to the system \eqref{helmotzequationn} satisfies the following estimate:
\begin{equation}\label{stabest}    \frac{c^*}{1+\omega^2} \Big (  \lVert \nabla u \rVert^2_{L^2(\Omega)} + \omega^2  \lVert u \rVert^2_{L^2(\Omega)} \Big )  \leq  \lVert (\Delta + \omega^2I)u \rVert^2_{L^2(\Omega)} + \Big\lVert \frac{\partial u}{\partial n} \Big\rVert^2_{L^2(\Gamma) },
    \end{equation}
     for all $ \omega \in \mathbb{R}$, where $c^* > 0$ that only depends on $\Gamma$ and $\Omega$. 
     \end{lemma}

By Theorem \ref{observabilitywavee}, there exists a positive and even function $\delta$, decreasing on 
$\mathbb R_+$, such that the following inequality holds 

 \begin{eqnarray} \label{Hautus conditionn}
 \delta(\lambda) \left \lVert z \right \rVert^{2}_{H_0^1(\Omega) \times L^2(\Omega)} \leq 
\left \lVert (A - \lambda I)z \right \rVert^{2}_{H_0^1(\Omega) \times L^2(\Omega)} + \left \lVert Cz \right \rVert^{2}_{L^{2}(\Gamma)}, \qquad \forall \lambda \in \mathbb{R}, \;  \forall z \in D(A).
 \end{eqnarray}

Following the proof of Theorem~\ref{observabilitywavee}, one can show that there exists a constant
$c_5=c_5(\Gamma,\|C_0\|)>0$ such that
\begin{equation} \label{constantc2}
\delta(\lambda)\geq \frac{c_5}{1+\lambda^2}, \qquad \text{for all } \lambda\in\mathbb{R}.
\end{equation}

It follows from inequalities \eqref{Hautus conditionn} and \eqref{constantc2} that the operator $C$ is spectrally coercive with
\[
 \widetilde \Psi(\lambda)=  \widetilde \varepsilon(\lambda)= \frac{c_5}{2(1+\lambda^2)}, \quad \forall \lambda \in \mathbb R.
\]
Therefore, by Theorem~\ref{spectralcrvv1}, the observability inequality \eqref{obscarre} holds with $ \mathtt T_0(\lambda) $  solution to 
\begin{equation} \label{eeer}
T\widetilde \varepsilon( \theta_1(\frac{1}{T} +\lambda)) \;= \;\theta_2.
\end{equation}

We next derive a simplified an upper bound $ \mathtt T(\lambda) $  of $ \mathtt T_0(\lambda) $. We first deduce from \eqref{eeer} that
$ \mathtt T_0(\lambda) \geq 2c_5^{-1}\theta_2$. Using the latter bound  in  equation \eqref{eeer}, we get 
\[
\mathtt T_0(\lambda)  \leq 2\theta_2(1+\theta_1^2(c_5(2\theta_2)^{-1} +\lambda)^2)\leq \mathtt T(\lambda)=  (c_1+ c_2\lambda^2),
\]
where $c_1= 2\theta_2(1+ 2c_5^2(2\theta_2)^{-2})$  and $c_2= 4\theta_2$. On the other hand, since 
$\mathtt  T(\lambda) \geq  2c_5^{-1}\theta_2$,
we have 

\begin{equation}
\theta_2  \widetilde \Psi( \theta_0(\frac{1}{T}+\lambda))  \geq \frac{c_5 \theta_2}{2(1+\theta_0^2 ( c_5 (2\theta_2)^{-1}+\lambda)^2)}\geq \Psi(\lambda)= \frac{c_4}{c_3+c_2\lambda^2}, \qquad \forall \lambda \in \mathbb R_+,
\end{equation} 
with $c_4= c_5\theta_2$, and $c_3 = 4 \theta_0^2c_5^2 (2\theta_2)^{-2}$.
\end{proof}

\begin{remark}
The approach developed in this section readily extends to a wide range of settings in which resolvent estimates of the form \eqref{stabestt1}  are available \cite{AA,BO,CL}.
\end{remark}
\section{Appendix}
\subsection{Proof of Theorem \ref{spectralcrv}}

    Let $z \in D(A) \setminus\{ 0 \}$ be fixed.
    We first assume that $C \in \mathcal L(D(A), Y)$ is spectrally coercive and prove that \eqref{resolineqq1} is satisfied. Let $\varepsilon, \, \Psi$ $\in \Sigma_d$ the functions appearing in the spectral coercivity of the operator $C$ in Definition 2.1, and consider the following possible two cases:
    
\noindent \textbf{{Case 1}:} Assume that 
 $$\frac{\lVert Az \rVert^{2}_{X}}{\lVert z \rVert^{2}_{X}} - \lambda_A^{2}(z) \leq   \varepsilon(\lambda_{|A|}(z)),$$ then by the spectral coercivity of $C$, we deduce that \eqref{resolineqq1}  holds.\\

\noindent \textbf{{Case 2}:} Assume that 
 \begin{equation} \label{iiiqqq}
 \frac{\lVert Az \rVert^{2}_{X}}{\lVert z \rVert^{2}_{X}} - \lambda_A^{2}(z) > \varepsilon(\lambda_{|A|}(z)).
\end{equation}
Now, take $$\phi(\lambda) = \frac{\lVert (A -\lambda I)z \rVert^{2}_{X}}{(\lambda - \lambda_A(z))^{2} + \varepsilon(\lambda_{|A|}(z))},$$ and let us study its variation. By simple calculation, we obtain 
 
 $$\phi^\prime(\lambda) = \frac{ 2(\lambda -\lambda_A(z)) \lVert z \rVert^{2}_{X} \left(\varepsilon(\lambda_{|A|}(z)) - \frac{\lVert Az \rVert^{2}_{X}}{\lVert z \rVert^{2}_{X}} + \lambda_A^2(z)\right)}{((\lambda - \lambda_A(z))^{2} + \varepsilon(\lambda_{|A|}(z)))^{2}}.$$

Hence $\phi$ attains a critical point at $\lambda = \lambda_A(z)$. We deduce from \eqref{iiiqqq} that $\phi$ exhibits a monotonic increase over $(-\infty, \lambda_A(z))$ and a monotonic decrease over $(\lambda_A(z), +\infty)$, indicating a global maximum at $\lambda_A(z)$. In particular, $\phi(\lambda)$ asymptotically approaches $\lVert z \rVert^{2}_{X}$ as $\lambda \to \pm\infty $. In conclusion, we have 

\begin{equation} \label{iiineq}  \frac{\lVert (A -\lambda I)z \rVert^{2}_{X}}{(\lambda - \lambda_A(z))^{2} + \varepsilon(\lambda_{|A|}(z))} \geq \|z\|^2_{X}, \end{equation}  for all $\lambda \in \mathbb{R},$ which provides 
the desired inequality  \eqref{resolineqq1}.\\

We now assume that \eqref{resolineqq1} holds, and we shall show that $C \in  \mathcal L(D(A), Y)$ satisfies the spectral coercivity in Definition 2.1. Let $\varepsilon$ and $\Psi$ $\in \Sigma_d$ be the functions appearing in \eqref{resolineqq1} and assume that $z \in D(A) \setminus \{0\}$ satisfies
\begin{equation} \label{pppoi}
\lVert (A -\lambda_A(z)I)z \rVert^{2}_{X}  \leq  \varepsilon (\lambda_{|A|}(z)) \lVert z \rVert^{2}_{X}.
\end{equation}
On the other hand  taking  $\lambda = \lambda_A(z) \in \mathbb{R}$  in  \eqref{resolineqq1}, leads to
\begin{equation} \label{pppoi2}
 \lVert z \rVert^{2}_{X} \leq \sup \Big\{ \frac{\lVert Cz \rVert^{2}_{Y}} {\Psi(\lambda_{|A|}(z))}  , \frac{\lVert (A -\lambda_A(z) I)z \rVert^{2}_{X}}{ \varepsilon(\lambda_{|A|}(z))} \Big \}. 
 \end{equation}

Combining  now \eqref{pppoi} and \eqref{pppoi2} implies 
$\lVert z \rVert^{2}_{X} \leq \frac{\lVert Cz \rVert^{2}_{Y}} {\Psi(\lambda_{|A|}(z))} $, which gives the desired result.

\subsection{Proof of Theorem \ref{spectralcrvv1}}
The proof follows Fourier techniques developed in \cite{BZ, AT}.\\
    
We further assume the spectral coercivity of $C$, which is equivalent to the resolvent inequality \eqref{resolineq} holding, and we shall establish weak observability.

Let $\chi \in C_0^\infty(\mathbb{R})$ be a cut off function with a
compact support in $(-1, 1)$. For $T>0$, we further denote 
 \begin{equation}\label{cutoff}
\chi_T(t) = \chi \left(\frac{t}{T}\right), \qquad  t\in \mathbb{R}.
\end{equation}

Let $z_{0}\in D(A)\setminus\{0\}$.
Set $z(t)=e^{itA}z_{0}$, $x=\chi_T z$ and $f=\dot x -iA x$.
Since $\dot z-iA z =0$,  we have $f= \dot \chi_T z$.
The Fourier transform of $f$ with respect to time is given by
$$\widehat{f}(\tau)=(i\tau-iA)\widehat{x}(\tau),$$

where $\widehat{x}(\tau)$ is the Fourier transform of $x(t)$.
Applying \eqref{resolineqq1} to $\widehat{x}(\tau)\in D(A)\setminus\{0\}$ for 
$\lambda= \tau$, we obtain
\begin{equation}\label{ineq}
\|\widehat x(\tau)\|_X^2 \leq \sup \left\{\frac{\|{C \widehat x(\tau)}\|^{2}_Y}{\psi(\lambda_{|A|}(\widehat x(\tau)))}, 
\frac{\|\widehat{f}(\tau)\|^{2}_X}{(\tau-\lambda_A(\widehat x(\tau)))^2+\varepsilon(\lambda_{|A|}(\widehat x(\tau)))} \right\}.
\end{equation}

We remark that since $\widehat x(\tau) \not= 0,$  the inequality \eqref{ineq} is well justified. Next, we study the variation of the frequency $\lambda_{|A|}(\widehat x(\tau))$ as a function of $\tau$. \\

To simplify the analysis we will make some assumptions 
 on the cut-off function $\chi(s)$.  
We further assume that  $\chi \in C_0^\infty(\mathbb{R})$  is even and satisfies
 the following inequalities: 
\begin{equation} \label{bbchi}
\chi \in H^1_0(-1,1),\;\;\;  \frac{\kappa_1}{1+\tau^2}
 \leq |\widehat \chi(\tau)|  \leq \frac{\kappa_2}{1+\tau^2},  \, \forall  \tau \in \mathbb{R},
\end{equation}
where $\kappa_2> \kappa_1>0$ are two fixed constants that do not depend on $\tau$ (see Appendix in \cite{AT} for the construction of such a function). 

\begin{theorem}\label{Tfrequency} 
Let $z_{0} \in D(A)\setminus\{0\}$, and let  $z(t) = e^{itA}z_0$, and let $\widehat{x}(\tau)$ be the Fourier transform of $x(t)= \chi_T(t) z(t)$, where $\chi_T(t)$ is the cut-off function defined by \eqref{cutoff},  and satisfying  the inequality \eqref{bbchi}. \\
 
 Then, there exists a constant $c_0= c_0(\chi)>0$ such that 
 the following inequality
\begin{equation} \label{lambdatau}
\lambda_1 \leq |\lambda_{|A|}(\widehat x(\tau))|\leq 4|\tau|+ c_0|\lambda_{|A|}(z_0)|, 
\end{equation}
holds for all $ \tau \in \mathbb{R}.$ 
\end{theorem}

\begin{proof}
Recall the expression of the frequency function:
\begin{equation} \label{lambda}
\lambda_{|A|}(\widehat x(\tau)) =  \langle |A|\widehat x(\tau) , \widehat x(\tau) 
\rangle_X \|\widehat x(\tau)\|_X^{-2},
\;\; \forall
\tau \in \mathbb{R}.  
\end{equation}
Let $z_0= \sum_{k\in \mathbb Z^*}  z_k \psi_k \in D(A)$.  Simple calculation leads to

\begin{equation}
\widehat x(\tau) =  \sum_{k\in \mathbb Z^*} \widehat \chi_T(\tau-\mu_k) z_k \psi_k.  
\end{equation}
Hence
\begin{equation} \label{ide}
\lambda_{A}(\widehat x(\tau)) =  \sum_{k\in \mathbb Z^*}  \lambda_{|k|} |
\widehat \chi_T(\tau-\mu_k)|^2 |z_k|^2 \left( \sum_{k\in \mathbb Z^*} 
|\widehat \chi_T(\tau-\mu_k)|^2 |z_k|^2 \right)^{-1}.  
\end{equation}

We first remark that  $\lambda_{|A|}(\widehat x(\tau))\geq \lambda_1$ for all $\tau \in \mathbb R$, and it tends to  $\lambda_{|A|}(z_0)$
when $\tau$ approaches  $\infty$. \\

Let $K \in \mathbb R_+$ be large enough, and set
$$
\sum_{k \in \mathbb Z^*}  |z_k|^2 =  \sum_{|\tau-\mu_k|\leq K} |z_k|^2 +\sum_{|\tau-\mu_k|> K}  
|z_k|^2
= \mathcal I_1+  \mathcal I_2.
$$
\begin{lemma} \label{lastlemma}
Let $K_\tau= |\tau|+r_0$  with
\begin{equation} \label{r0}
r_0= 3\lambda_{|A|}(z_0).
\end{equation}
Then  the following inequality holds
\begin{equation} \label{ii}
2\mathcal I_2 \leq \mathcal I_1, \;\; \; \textrm{  for all  } K \geq K_\tau.
\end{equation}
\end{lemma}
\proof 
We claim that  that there exists $r_0>0$ large enough such that 
\begin{equation} \label{ll1}
 2\sum_{\lambda_{|k|}> r}  |z_k|^2 \leq 
  \sum_{\lambda_{|k|} \leq r}  |z_k|^2, \qquad \forall r \geq r_0. 
\end{equation}
or equivalently 

$$
3\sum_{\lambda_{|k|}> r}  |z_k|^2  \leq \|z_0\|_X^2.
 $$

In fact, we have 
\begin{equation} \label{pp}
\sum_{\lambda_{|k|}> r}  |z_k|^2 <  \frac{1}{r}
\sum_{\lambda_{|k|}> r_0} \lambda_{|k|}  |z_k|^2  \leq 
\frac{\lambda_{|A|}(z_0)}{r}\|z_0\|_X^2,
\end{equation}
and   inequality \eqref{ll1}  holds for all $ r\geq r_0$.\\

 Now by taking $K=|\tau| +r_0$,  and using the bounds \eqref{bbchi} with
 $\widehat \chi_T(s)= T\widehat \chi(Ts)$ in mind, 
we get  
\begin{equation}\label{ff1}2\mathcal I_2 \leq  2\sum_{\lambda_{|k|}> K+\tau}  |z_k|^2,
\end{equation}

\begin{equation}\label{ff2} \mathcal I_1 \geq \sum_{\lambda_{|k|}\leq K+|\tau|}  |z_k|^2.
\end{equation}
Since $K+\tau \geq r_0$, inequalities \eqref{ll1}, \eqref{ff1} and  \eqref{ff2} imply
\begin{equation}
2\mathcal I_2 \leq \sum_{\lambda_{|k|}\leq K+\tau} 
\lambda_{|k|} |z_k|^2 \leq  \mathcal I_1.
\end{equation}
Then,  inequality \eqref{ii} is valid for all $ K\geq K_\tau= |\tau|+r_0$. \\
\endproof

Back to the proof of the Theorem \ref{Tfrequency}. Using the fact that 

\begin{equation} \label{iii}
\sum_{\lambda_{|k|} \leq K}  
\widehat \chi_T(\tau-\mu_k)|^2 |z_k|^2 \leq \|\widehat x(\tau)\|_X^{2},
\end{equation}
 we have 
\begin{align*}
 \lambda_{|A|}(\widehat x(\tau)) \leq \ \ &   \left(\sum_{|\tau-\mu_k|\leq K} \lambda_{|k|} 
|\widehat \chi_T(\tau-\mu_k)|^2 |z_k|^2\right) \left( \sum_{|\tau-\mu_k|\leq K}
|\widehat \chi_T(\tau-\mu_k)|^2 |z_k|^2 \right)^{-1} \\
& + \left(\sum_{|\tau-\mu_k|>K} \lambda_{|k|} |
\widehat \chi_T(\tau-\mu_k)|^2 |z_k|^2\right) \left( \sum_{|\tau-\mu_k|\leq K}
|\widehat \chi_T(\tau-\mu_k)|^2 |z_k|^2 \right)^{-1}
&= \mathcal J_1 +\mathcal J_2.
\end{align*}

On the other hand we have
\begin{equation} \label{j1}
\mathcal J_1 \leq (K+|\tau|).
\end{equation}
In addition, using again the bounds \eqref{bbchi},  we obtain
\begin{equation} \label{j2}
\mathcal J_2 \leq  \frac{\kappa_2}{\kappa_1}\sum_{|\tau-\mu_k| > K} \lambda_{|k|}
|z_k|^2 \left( \sum_{_{|\tau-\mu_k| \leq K}  }  |z_k|^2 \right)^{-1}. 
\end{equation}

Taking  $K \geq K_\tau$, Lemma \ref{lastlemma}
leads to  
\begin{equation*} \label{bb}
\mathcal J_2 \leq  \left( \frac{2}{3} \|\widehat x(\tau)\|_X^{2} \right)^{-1}. 
\end{equation*}

Hence

\begin{equation} \label{jj2}
\mathcal J_2 \leq  \frac{3\kappa_2}{2\kappa_1}  \lambda_{|A|} (z_0).
\end{equation}
Combining inequalities \eqref{j1} and \eqref{jj2}, we get 
$$
 \lambda_{|A|}(\widehat x(\tau)) \leq K+ |\tau|  + \frac{3\kappa_1}{2\kappa_2} \lambda_{|A|}(z_0).
 $$
 for all  $ K \geq K_\tau$.\\
 
Consequently, the proof is achieved by taking  
$c_0=   3+ \frac{3\kappa_2}{2\kappa_1}$
\end{proof}

\begin{remark}
The upper bound of $\lambda_{|A|}(\widehat x(\tau))$ obtained 
in Theorem~\ref{Tfrequency} is not optimal since 
$\lambda_{|A|}(\widehat x(\tau)) = |\lambda_{|k|}|= \lambda_{|A|}(z_0)$ if $z_0 = \psi_k$.
Moreover when $\lambda_{max}(z_0) = \max \{\lambda_{|k|}, \; k \in \mathbb Z^*,\; 
\langle z_0,\psi_k\rangle_X \not= 0\} < \infty$, we can easily show that  
$ \lambda_{|A|}(\widehat x(\tau)) \leq \lambda_{max}(z_0) $.  We remark that  in both cases the bounds of $\lambda_{|A|}(\widehat x(\tau)) $ are independent of the Fourier frequency $\tau$.
 \end{remark}

 \begin{lemma} Let $c_0^\prime =  \frac{\|\dot \chi\|_{L^2(-1, 1)}}{\|\chi\|_{L^2(-1, 1)}},$
  $z_{0}\in D(A)\setminus\{0\}$, and let  $z(t) = e^{itA}z_0$, and let $\widehat{x}(\tau)$ be the Fourier transform of $x(t)= \chi_T(t) z(t)$, where $\chi_T(t)$ is the cut-off function defined by \eqref{cutoff}. \\
 
 Then,   the following inequality 
 \begin{equation} \label{inn1}
 \left(1- \frac{1}{R}\left( \frac{c_0^\prime}{T}  +\lambda_{|A|}(z_0)
  \right)\right) \|z_0\|_X^2 \leq \|\chi\|_{L^2(-1, 1)}^{-2} \int_{-R}^R\|\widehat{x}(\tau)\|_X^2 d\tau 
 \end{equation}
 holds for all $R >  \frac{c_0^\prime}{T}  +\lambda_{|A|}(z_0).$
 
 \end{lemma}
\begin{proof}
Recall that 
$\dot x = f+iA x$ where $f= \dot \chi_T z$. By integration by parts we then have
\[
\widehat{x}(\tau) = -\frac{i}{\tau} \left(\widehat{f}(\tau)+ iA\widehat{x}(\tau) \right).
\]
Consequently 
\[
\|\widehat{x}(\tau)\|_X^2  =
\langle -\frac{i}{\tau} \left(\widehat{f}(\tau)+ iA\widehat{x}(\tau) \right), \widehat{x}(\tau)\rangle_X.
\]
Then for any $R>0$, by Fourier-Plancherel Theorem, we have
$$
  \|\chi\|_{L^2(-1, 1)}^2 \|z_0\|_X^2 
   \leq \int_{-R}^R\|\widehat{x}(\tau)\|_X^2 d\tau+
  \frac{1}{R}\left( \frac{1}{T} \|\dot \chi\|_{L^2(-1, 1)}\|\chi\|_{L^2(-1, 1)} +\lambda_{|A|}(z_0)\|\chi\|_{L^2(-1, 1)}^2 
   \right)\|z_0\|_X^2.
$$
Hence for $R$ large enough we have 
$$
 \left(1- \frac{1}{R}\left( \frac{1}{T} 
 \frac{\|\dot \chi\|_{L^2(-1, 1)}}{\|\chi\|_{L^2(-1, 1)}} +\lambda_{|A|}(z_0)
  \right)\right) \|z_0\|_X^2 
   \leq \|\chi\|_{L^2(-1, 1)}^{-2}\int_{-R}^R\|\widehat{x}(\tau)\|_X^2 d\tau,
 $$
which finishes the proof of the lemma.

\end{proof}

Back now to the proof of the theorem. Combining inequalities \eqref{ineq} and 
 \eqref{inn1}, we find
 
 \begin{equation} \label{mm1}
    \left(1- \frac{1}{R}\left( \frac{c_0}{T}  +\lambda_{|A|}(z_0)
  \right)\right) \|z_0\|_X^2 \leq \|\chi\|_{L^2(-1, 1)}^{-2} \left( \int_{-R}^R
   \frac{\|{C \widehat x(\tau)}\|^{2}_Y}{\psi(\lambda_{|A|}(\widehat x(\tau)))} d\tau + 
 \int_{-R}^R\frac{\|\widehat{f}(\tau)\|^{2}_X}{\varepsilon(\lambda_{|A|}(\widehat x(\tau)))}d\tau\right).
\end{equation}
 
 Applying the upper bound $\lambda_{|A|}(\widehat x(\tau))$ derived in 
 Theorem~\ref{Tfrequency}, and considering the monotony of the functions
 $\psi$ and $\varepsilon$ in $\Sigma_d$, we obtain 
  \begin{align*}
    \left(1- \frac{1}{R}\left( \frac{c_0^\prime}{T}  +\lambda_{|A|}(z_0)
  \right)\right) \|z_0\|_X^2 \leq \frac{1}{\psi(4R+ c_0\lambda_{|A|}(z_0))}  
  \frac{\|\chi\|_{L^\infty(-1,1)}^2}{\|\chi\|_{L^2(-1, 1)}^{2}}
  \int_{0 }^T|{C z(t)}\|^{2}_Yd\tau \\
  +
 \frac{1}{T\varepsilon(4R+ c_0\lambda_{|A|}(z_0))} 
 \frac{\|\dot \chi\|_{L^2(-1, 1)}^2}{\|\chi\|_{L^2(-1, 1)}^{2}}\|z_0\|_X^2,
\end{align*}
 for all $R >  \frac{c_0^\prime}{T}  +\lambda_{|A|}(z_0)$.\\
 
 Now, by taking $R= 2\left(\frac{c_0}{T}  +\lambda_{|A|}(z_0)\right)$, and 
 $\theta_0= \max(c_0^\prime, 8+c_0)$, we find
 
 \begin{align*}
  \left(1-   \frac{2}{T\varepsilon\left(\theta_0\left(\frac{1}{T}+\lambda_{|A|}(z_0)\right)\right)} 
 \frac{\|\dot \chi\|_{L^2(-1, 1)}^2}{\|\chi\|_{L^2(-1, 1)}^{2}}\right)\|z_0\|_X^2 
 \leq \frac{2}{\psi\left(\theta_0\left(\frac{1}{T}+\lambda_{|A|}(z_0)\right)\right)}  
  \frac{\|\chi\|_{L^\infty(-1,1)}^2}{\|\chi\|_{L^2(-1, 1)}^{2}}
  \int_{0 }^T\|{C z(t)}\|^{2}_Ydt.
\end{align*}

Let $\theta_1=  \frac{4\|  \chi\|_{L^2(-1, 1)}^{2}}{\| \dot \chi\|_{L^\infty(-1,1)}^2}$, and 
$\theta_2= \frac{4\|\chi\|_{L^2(-1,1)}^2}{\| \chi\|_{L^\infty(-1, 1)}^{2}}
$.  \\

Then, for $T\varepsilon(4R+ c_0\lambda_{|A|}(z_0))\geq  \theta_1$,  we finally get the wanted estimate:

 \begin{equation}  \label{fee}
  \theta_2\psi\left(\theta_0\left(\frac{1}{T}+\lambda_{|A|}(z_0)\right)\right)   \|z_0\|_X^2 \leq 
   \int_{0 }^T\|{C z(t)}\|^{2}_Ydt.
\end{equation}
Simple calculation shows that the function $T\mapsto 
T\varepsilon \left(\theta_0\left(\frac{1}{T}+\lambda_{|A|}(z_0)\right)\right)$ is increasing,
tends to infinity when $T$ approaches $+\infty$, and tends  to $0$ when $T$ 
approaches $0$. Then there exists a unique value $T(\lambda_{|A|}(z_0))>0$ that solves
the equation \eqref{Eq:T}. In addition, the function  $\lambda \mapsto T(\lambda)$ is
 increasing. Finally, the inequality \eqref{fee} is valid for all $T \geq T(\lambda_{|A|}(z_0))$. \\

Now, we shall prove the converse.
We further assume that the weak observability inequality \eqref{observability4}  holds for some fixed  $\psi$ and $\varepsilon$ in $\Sigma_d$. Our goal now is to show that  $C$ is indeed  spectrally coercive.\\

Let $z_0 \in D(A^2)$, and  $x_0:= (iA-i\tau I)z_0 $ for some $\tau \in \mathbb{R}$. Define $x(t)= e^{itA}x_0$ and
 $z(t) = e^{itA}z_0$. \\
 
  A forward computation shows  that $z(t)$ solves the following 
 $$
 \dot z(t) -i\tau z(t) = x(t), \;\;\forall t\in \mathbb{R}_+^*,\\
 z(0) = z_0.
 $$
 
 Then 
 
 \[
 z(t) = e^{i\tau t} z_0 +\int_{0}^t e^{i\tau(t-s)} x(s) ds.
 \]
 Applying now the observability operator both sides gives
 $$
 Cz(t) = e^{i\tau t} Cz_0+ \int_{0}^t e^{i\tau(t-s)} Cx(s) ds,
$$
 whence 
  $$
 \|Cz(t)\|_Y^2 \leq 2\|Cz_0\|_Y^2+2 \int_{0}^t \|Cx(s)\|_Y^2ds. 
 $$

Integrating the inequality above both sides over $(0, T)$, we obtain
$$
\int_0^T \|Cz(t)\|_Y^2 dt \leq 2T\|Cz_0\|_Y^2+2 T \int_0^T\| Cx(s) \|_Y^2 ds.
$$

We deduce from the admissibility assumption \eqref{observability3}
 that 
$$
 \int_0^T \|Cz(t)\|_Y^2 dt \leq 2T\|Cz_0\|_Y^2+2 T C_T\|(A-\tau I)z_0\|^2_X.
 $$

Applying the weak observability inequality \eqref{observability4} for $T= T(\lambda_{|A|}(z_0))$,
leads to 
\begin{align*}
&\theta_2\psi\left(\theta_0\left(\frac{1}{T(\lambda_{|A|}(z_0))}+\lambda_{|A|}(z_0)\right)\right)   \|z_0\|_X^2 \\
 &\leq 2T(\lambda_{|A|}(z_0))\|Cz_0\|_Y^2+2 T(\lambda_{|A|}(z_0)) C_{T(\lambda_{|A|}(z_0))}\|(A-\tau I)z_0\|^2_X,
\end{align*}
for all $\tau \in \mathbb R$. \\ 

Since $T(\lambda) \geq T_0 = T(0),$ for all $\lambda \geq 0$, we have
\begin{align*}
&\theta_2\psi\left(\theta_0\left(\frac{1}{T_0}+\lambda_{|A|}(z_0)\right)\right)   \|z_0\|_X^2 \\
&\leq 2T(\lambda_{|A|}(z_0))\|Cz_0\|_Y^2+2 T(\lambda_{|A|}(z_0)) C_{T(\lambda(z_0))}\|(A-\tau I)z_0\|^2_X.
\end{align*}

Taking $\tau = \lambda_A(z_0)$ in the previous inequality implies 
\begin{align*}
&\frac{\theta_2}{2 T(\lambda_{|A|}(z_0)) C_{T(\lambda_{|A|}(z_0))}}
\psi\left(\theta_0\left(\frac{1}{T_0}+\lambda_{|A|}(z_0)\right)\right)   \|z_0\|_X^2 \\
&\leq \frac{1}{C_{T(\lambda_{|A|}(z_0))}} \|Cz_0\|_X^2+\|(A-   \lambda_{A}(z_0)I)z_0\|^2_X.
\end{align*}

Let
$$
\widetilde \psi(\lambda) = \frac{\theta_2}{4 T(\lambda) }\psi\left(\theta_0\left(\frac{1}{T_0}+\lambda\right)\right),
$$\\
$$
\widetilde \varepsilon(\lambda) = \frac{\theta_2}{4 T(\lambda) C_{\lambda}} \psi \left(\theta_0\left(\frac{1}{T_0}+\lambda \right)\right).
$$

We deduce from the monotonicity properties  of $\psi(\lambda), C_\lambda, $ and $T(\lambda)$ that 
 $\widetilde \psi(\lambda),\, \widetilde \varepsilon(\lambda)
\in \Sigma_d $. \\

Consequently $C$ becomes  spectrally coercive with the functions
$\widetilde \psi(\lambda),\, \widetilde \varepsilon(\lambda)$, that is 

 \begin{eqnarray*}
 0\leq \frac{\|Az\|_X^2}{\|z\|_X^2} -\lambda_A^2(z) \leq \widetilde \varepsilon(\lambda_{|A|}(z)),
\end{eqnarray*}
implies
\begin{eqnarray*}
\|Cz\|_Y^2 \geq \widetilde \psi(\lambda_{|A|}(z)) \|z\|_{X }^2,
\end{eqnarray*}
which finishes the proof of the Theorem.

\end{document}